\documentclass[11pt,oneside,reqno]{amsart}

\usepackage{amsmath,amssymb,amsthm,textcomp}
\usepackage{amsfonts,graphicx}
\usepackage[mathscr]{eucal}
\usepackage{color}
\usepackage{csquotes}
\usepackage{enumitem}
\usepackage{
mathtools}
\numberwithin{equation}{section}

\theoremstyle{definition}

\numberwithin{equation}{section}

\newtheorem{theorem}{\bf Theorem}[section]
\newtheorem{remark}{\bf Remark}[section]
\newtheorem{proposition}{Proposition}[section]

\newtheorem{definition}{Definition}[section]
\newtheoremstyle
{remarkstyle}
{}
{11pt}
{}
{}
{\bfseries}
{:}
{     }
{\thmname{#1} \thmnumber{#2} }

\theoremstyle{remarkstyle}

\begin{document}
	\title{On Time-Changed Linear Birth-Death Process with Immigration at Extinction}
	\author{Kuldeep Kumar Kataria}
	\address{K.K. Kataria, Department of Mathematics,
		 Indian Institute of Technology Bhilai, Durg 491002, INDIA.}
	 \email{kuldeepk@iitbhilai.ac.in}
	 	\author[Pradeep Vishwakarma]{Pradeep Vishwakarma}
	 \address{P. Vishwakarma, Department of Mathematics,
	 	Indian Institute of Technology Bhilai, Durg, 491002, INDIA.}
	 \email{pradeepv@iitbhilai.ac.in}
	
	\subjclass[2010]{Primary : 60J27; Secondary: 60J20}
	
	\keywords{linear birth-death process, inverse stable subordinator, fractional derivative}
	\date{\today}
	
	\maketitle
	\begin{abstract}
	 We study an immigration effect in the time-changed linear birth-death process where the immigration occurs only if the population goes extinct. We call this process as the time-fractional linear birth-death process with immigration (TFLBDPwI). Its transient probabilities are obtained using the Adomian decomposition method. In a particular case, we discuss the accuracy of approximation for the obtained transient probabilities. Also, the first two moments  of TFLBDPwI are derived. Later, we consider some particular cases of the TFLBDPwI and study their distributional properties in detail.
	\end{abstract}
	\section{Introduction}\label{sec1}
		 Birth-death process is a continuous-time Markov process used to model population growth over time. It has wide applications in both theoretical and practical contexts, such as modeling queues in service systems, epidemics, inventory systems, and the evolution of biological populations. In the linear birth-death process (see \cite{Feller1964}, p. 456), the state zero is an absorbing state, that is, once the population get extinct it cannot be revived again. Thus, the linear birth-death process is not a suitable model for applications dealing with real-life situations where the population starts from zero. In such cases, a birth-death process with immigration is more appropriate. Recently, Kataria and Vishwakarma \cite{Kataria2025} studied an immigration effect in the time-changed linear birth-death process where the immigration can happen at any state. For some recent works on the birth-death processes with immigration, we refer the reader to \cite{Crescenzo2016}, \cite{Dessalles2018} and \cite{Kataria2024}.

	Here, we consider the effect of immigration in  the linear birth-death process only when the population get extinct. We call this process as the linear birth-death process with immigration (LBDPwI) and denote it by $\{N(t)\}_{t\ge0}$. Let $\lambda=\alpha,\ \lambda_n=n\lambda$ for all $n\ge1$ be the birth rates and $\mu_n=n\mu$ for all $n\ge1$ be the death rates. Here, $\lambda $ and $\mu$ are positive constants and $\alpha>0$ is the rate of immigration.	
	Then, the state probabilities $p(n,t)=\mathrm{Pr}\{N(t)=n\}$, $n\ge0$ of LBDPwI solve the following system of differential equations:
	\begin{equation}\label{pr0}
		\frac{\mathrm{d}}{\mathrm{d}t}p(n,t)=\begin{cases}
			-\alpha p(0,t)+\mu p(1,t),\ n=0,\vspace{0.1cm}\\
			-(\lambda+\mu)p(1,t)+\alpha p(0,t)+2\mu p(2,t),\ n=1,\vspace{0.1cm}\\
			-n(\lambda+\mu)p(n,t)+(n-1)\lambda p(n-1,t)+(n+1)\mu p(n+1,t),\ n\ge2,
		\end{cases}
	\end{equation}
	where $p(n,t)=0$ for all $n<0$. If no individual is present at time $t=0$ then the initial conditions are $p(0,0)=1$ and $p(n,0)=0$ for all $n\ne0$. From (\ref{pr0}), the probability generating function (pgf) $G(u,t)=\sum_{n=0}^{\infty}u^np(n,t)$, $|u|\leq1$ of LBDPwI solves 
	\begin{equation}\label{pgf1}
		\frac{\partial}{\partial  t}G(u,t)=(\lambda u-\mu)(u-1)\frac{\partial}{\partial u}G(u,t)+\alpha(u-1)p(0,t)
	\end{equation}
with $G(u,0)=1.$

	Over the past two decades, many random time-changed and time-fractional growth processes, for example, the fractional Poisson process (see \cite{Beghin2009}, \cite{Leskin2003}, \cite{Meerschaert2011}), the fractional pure birth process (see \cite{Orsingher2010}), and the fractional linear birth-death process (see \cite{Orsingher2011}), have been studied. A key feature of these fractional processes is global memory, which characterizes many real-world systems.

	In this paper, we introduce and study a time-changed variants of the LBDPwI $\{N(t)\}_{t\ge0}$ with an inverse $\nu$-stable subordinator $\{L_\nu(t)\}_{t\ge0}$ (for definition see Section \ref{insubdef}). We call it as the time-fractional linear birth-death process with immigration (TFLBDPwI) and denote it by $\{N^\nu(t)\}_{t\ge0}$. It is defined as $N^\nu(t)\coloneqq N(L_\nu(t)),\ 0<\nu\leq1$, where $N(t)$ and $L_\nu(t)$ are independent and $N^1(t)=N(t)$. It is shown that the state probabilities $p^\nu(n,t)=\mathrm{Pr}\{N^\nu(t)=n\}$ of TFLBDPwI satisfy the following system of fractional differential equations:
	\begin{equation}\label{frdiffequ}
		\frac{\mathrm{d}^\nu }{\mathrm{d}t^\nu}p^\nu(n,t)=\begin{cases}
			-\alpha p^\nu(0,t)+\mu p^\nu(1,t),\ n=0,\vspace{0.1cm}\\
			-(\lambda+\mu)p^\nu(1,t)+\alpha p^\nu(0,t)+2\mu p^\nu(2,t),\ n=1,\vspace{0.1cm}\\
			-n(\lambda+\mu)p^\nu(n,t)+(n-1)\lambda p^\nu(n-1,t)+(n+1)\mu p^\nu(n+1,t),\ n\ge2,
		\end{cases}
	\end{equation} 
	with the initial conditions $p^\nu(0,0)=1$ and $p^\nu(n,0)=0$ for all $n\ne0$.
	Here,
	\begin{equation*}
		\frac{\mathrm{d^\nu} }{\mathrm{d}t^\nu}f(t)=\begin{cases}
			\frac{1}{\Gamma(1-\nu)}\int_{0}^{t}\frac{(\mathrm{d}/\mathrm{d}s)f(s)}{(t-s)^\nu}\,\mathrm{d}s,\ 0<\nu<1,\vspace{0.1cm}\\
		 ({\mathrm{d}}/{\mathrm{d}t})f(t),\ \nu=1,
		\end{cases}
	\end{equation*}
 is the Caputo fractional derivative whose Laplace transform is given by  (see \cite{Kilbas2006})

	\begin{equation}\label{frderlap}
		\int_{0}^{\infty}e^{-zt}\left(\frac{\mathrm{d}^\nu}{\mathrm{d}t^\nu}f(t)\right)\,\mathrm{d}t=z^{\nu}\int_{0}^{\infty}e^{-zt}{f}(t)\,\mathrm{d}t-z^{\nu-1}f(0),\ \nu>0.
	\end{equation}

	In the last section, we consider two particular cases of the LBDPwI and study their time-fractional variants. First, we introduce and study the time-fractional linear birth process with immigration and then the time-fractional linear death process with immigration. For both the processes, we obtain the explicit form of their state probabilities, means and variances.

		  As the TFLBDPwI has immigration effect, it has potential applications to model phenomena with rapidly changing conditions in which the state zero is not observable. For example, it can be used to model the rapid spread of a virus in a particular region. Initially, no individual in the population is infected with the virus however at some point of time someone gets infected due to the migration from outside. Here, we assume that every individuals present within the population are subject to an identical birth and death rates. It is done to prevent variation in the population wellness.
		  
	 Next, we give a brief introduction of the Adomian decomposition method (ADM).
\subsection{Adomian decomposition method } 
It is a numerical method for solving functional equations. Let us consider the following functional equation:
 \begin{equation}\label{fne}
 	x=f+N(x),
 \end{equation}
 where $N$ is a non-linear operator and $f$ is a known function. In ADM, it is assumed that the solution $x$ of (\ref{fne}) and operator $N$ can be expressed in terms of absolutely convergent series $x=\sum_{k=0}^{\infty}x_k$ and $N(x)=\sum_{k=0}^{\infty}A_k(x_0,x_1,\dots,x_{k})$, respectively, where $A_k$ is called the $k$th Adomian polynomial in $x_0,x_1,\dots,x_k$ (see \cite{Adomian1984}-\cite{Adomian1994}, \cite{Rach1984}). So, (\ref{fne}) can be rewritten as follows:
 \begin{equation*}
 \sum_{k=0}^{\infty}x_k=f+\sum_{k=0}^{\infty}A_k(x_0,x_1,\dots,x_{k}).
 \end{equation*}
 In ADM, the series components are obtained by using the following recurrence relation:
 \begin{equation*}
 	x_0=f\ \ \text{and}\ \ x_k=A_{k-1}(x_0,x_1,\dots,x_{k-1}),\ k\ge1.
 \end{equation*}
 Note that determining the Adomian polynomials is a key component in this method. Adomian \cite{Adomian1985} introduced a method to derive these polynomials by introducing a parameterization of the solution $x$. However, if $N$ is a linear operator, for example, $N(x)=x$ then $A_n$'s are simply $x_n$. For some recent work related to the Adomian polynomials, we refer the reader to \cite{Duan2010}, \cite{Duan2011} and \cite{Kataria2016}.
	
	\section{Time-fractional LBDPwI} \label{sec2}
	Here, we introduce a time-changed version of the LBDPwI where the time is changed according to an inverse stable subordinator. First, we recall the definitions of stable subordinator and its inverse.
	 \subsection{Subordinator and its inverse}\label{insubdef}
	A Subordinator is  non-decreasing L\'evy process (see \cite{Applebaum2004}).
	
	Let $\{S_\nu(t)\}_{t\ge0}$, $\nu\in(0,1)$ be a subordinator  with $\mathbb{E}e^{-zS_\nu(t)}=e^{-z^\nu t}$, $z>0$. Then, it is called $\nu$-stable subordinator.
	The first passage time $L_\nu(t)=\inf\{s>0:S_\nu(s)>t\}$  of $\nu$-stable subordinator is called an inverse $\nu$-stable subordinator. Its Laplace transform is given by (see \cite{Meerschaert2011})
	\begin{equation}\label{subdenlap1}
		\int_{0}^{\infty}e^{-zt}\mathrm{Pr}\{L_\nu(t)\in\mathrm{d}x\}=z^{\nu-1}e^{-z^\nu x}\,\mathrm{d}x,\ z>0.
	\end{equation} 
	
	Let $\{N(t)\}_{t\ge0}$ be the LBDPwI as defined in Section \ref{sec1}. We consider a time-changed process $\{N^\nu(t)\}_{t\ge0},\ 0<\nu\leq1$ defined as follows:
	\begin{equation}\label{timec1}
		N^\nu(t)\coloneqq N(L_\nu(t)),
	\end{equation}
	where the inverse $\nu$-stable subordinator $\{L_\nu(t)\}_{t\ge0}$  is independent of the LBDPwI and $N^1(t)=N(t)$ for all $t\ge0$. We call this process as the time-fractional linear birth-death process with immigration (TFLBDPwI). 
	
	The following result provides the system of differential equations that governs the state probabilities of TFLBDPwI.
	\begin{theorem}\label{thm3.1}
	The state probabilities $p^\nu(n,t)=\mathrm{Pr}\{N^\nu(t)=n\}$, $n\ge0$ of TFLBDPwI solve the system of differential equations given in (\ref{frdiffequ}).	
	\end{theorem}
	\begin{proof}
		For $n\ge0$ and $t\ge0$, from (\ref{timec1}), we have
		\begin{equation}\label{1}
			p^\nu(n,t)=\mathrm{Pr}\{N(L_\nu(t))=n\}=\int_{0}^{\infty}p(n,x)\mathrm{Pr}\{L_\nu(t)\in\mathrm{d}x\}.
		\end{equation}
		 Let $\tilde{p}(n,z)=\int_{0}^{\infty}e^{-zt}p(n,t)\,\mathrm{d}t$ be the Laplace transform of state probabilities of LBDPwI. Then, on taking the Laplace transform on both sides of (\ref{1}) and using (\ref{subdenlap1}), we have
		\begin{align}
			\tilde{p}^\nu(n,z)
			&=\int_{0}^{\infty}p(n,x)\left(\int_{0}^{\infty}e^{-zt}\mathrm{Pr}\{L_\nu(t)\in\mathrm{d}x\}\,\mathrm{d}t\right)\,\mathrm{d}x\nonumber\\
			&=z^{\nu-1}\int_{0}^{\infty}p(n,x)e^{-xz^\nu}\,\mathrm{d}x=z^{\nu-1}\tilde{p}(n,z^\nu),\ n\ge0.\label{lap1}
		\end{align}		 
		 On taking the Laplace transform on both sides of (\ref{pr0}), we get
		\begin{equation*}
			z\tilde{p}(n,z)-p(n,0)=-(\lambda_n+\mu_n)\tilde{p}(n,z)+\lambda_{n-1}\tilde{p}(n-1,z)+\mu_{n+1}\tilde{p}(n+1,z),\ n\ge0,
		\end{equation*}
		where $\lambda_0=\alpha,\ \lambda_n=n\lambda,\ \mu_n=n\mu$ for all $n\ge1$ and $\mu_0=\lambda_{-1}=0$. So, 
		\begin{multline}\label{a}
			z^{2\nu-1}\tilde{p}(n,z^\nu)-z^{\nu-1}p(n,0)\\=-z^{\nu-1}(\lambda_n+\mu_n)\tilde{p}(n,z^\nu)+z^{\nu-1}\lambda_{n-1}\tilde{p}(n-1,z^\nu)+z^{\nu-1}\mu_{n+1}\tilde{p}(n+1,z^\nu),\ n\ge0.
		\end{multline}
		On substituting (\ref{lap1}) in (\ref{a}) and using $p^\nu(n,0)=p(n,0)$ for all $n\ge1$, we have
		\begin{equation*}
		{z^\nu\tilde{p}^\nu(n,z)-z^{\nu-1}p^\nu(n,0)}=-(\lambda_n+\mu_n)\tilde{p}^\nu(n,z)+\lambda_{n-1}\tilde{p}^\nu(n-1,z)+\mu_{n+1}\tilde{p}^\nu(n+1,z),
		\end{equation*}
		 which on taking the inverse Laplace transform and using (\ref{frderlap}) yields
		\begin{equation*}
			\frac{\mathrm{d}^\nu }{\mathrm{d}t^\nu}p^\nu(n,t)=-(\lambda_n+\mu_n)p^\nu(n,t)+\lambda_{n-1}p^\nu(n-1,t)+\mu_{n+1}p^\nu(n+1,t),\ n\ge0.
		\end{equation*}
		This completes the proof.
	\end{proof}

	\begin{remark}\label{thm1}
			Let $\{T_{2\nu}(t)\}_{t\ge0}$ be a random process whose distribution function is the folded solution of the following fractional diffusion equation:
		\begin{equation*}
			\frac{\partial^{2\nu}g(x,t)}{\partial t^{2\nu}}=\frac{\partial^2g(x,t)}{\partial x^2},\ 0<\nu\leq1,\ x\in\mathbb{R}
		\end{equation*}
		with $g(x,0)=\delta(x),\ 0<\nu\leq1$ and ${\partial g(x,t)}/{\partial t}\big|_{t=0}=0,\ 1/2<\nu\leq1.$ Then, from Theorem 3.1 in \cite{Meerschaert2011}, it follows that $L_\nu(t)\overset{d}{=}T_{2\nu}(t)$ for $0<\nu<1$, where $\overset{d}{=}$ denotes the equality in distribution. Thus, the LBDPwI and TFLBDPwI satisfies the following time-changed relationship:
		\begin{equation*}
			N^\nu(t)\overset{d}{=}N(T_{2\nu}(t)),
		\end{equation*}
		where $N(t)$ is independent of $T_{2\nu}(t)$. A similar relationship holds for some other fractional processes, for example, the fractional Poisson process (see \cite{Meerschaert2011}), the fractional birth process (see \cite{Orsingher2010}) and the fractional linear birth-death process (see \cite{Orsingher2011}). 	
	\end{remark}
	
\subsection{Statistical properties of TFLBDPwI}
In this section, we derive the expression for state probabilities of TFLBDPwI and study some of its distributional properties. Note that there are no nonlinear terms involved in the system of differential equations corresponding to the TFLBDPwI. Thus, the ADM can be used to solve (\ref{frdiffequ}) to obtain the state probabilities of TFLBDPwI. First, we recall the definition of fractional integral and one related result.

\begin{definition}
	Let $f(\cdot)$ be an integrable function. Then, the Riemann-Liouville fractional integral is defined as follows (see \cite{Kilbas2006}):
	\begin{equation*}
		I_t^\nu(f(t))=\frac{1}{\Gamma(\nu)}\int_{0}^{t}\frac{f(s)\,\mathrm{d}s}{(t-s)^{1-\nu}},\ \nu>0,
	\end{equation*}
	where $I_t^\nu$ denotes the fractional integral operator of order $\nu$.
\end{definition} 
	For $\delta>0$, we have the following integral:
	\begin{equation}\label{lemma1}
	I_t^\nu(t^{\delta-1})=\frac{\Gamma(\delta)t^{\delta+\nu-1}}{\Gamma(\delta+\nu)},\ \nu>0.
	\end{equation}

Note that the Riemann-Liouville integral operator $I_t^\nu$, $\nu>0$ is linear. So, the Adomian polynomial $A_n$'s associated with it are given by $A_n(x_0(t),x_1(t),\dots,x_n(t))=I_t^\nu(x_n(t))$.

 Now, on applying $I_t^\nu$ on both sides of (\ref{fbpwiequ}) and substituting $p_{0}^\nu(n,t)=\sum_{k=0}^{\infty}p_{k}^\nu(n,t)$ for all $n\ge0$, we get
{\footnotesize\begin{equation}\label{tflbdpwiadm}
	\sum_{k=0}^{\infty}p_{k}^\nu(n,t)=\begin{cases}
		p^\nu(0,0)+\sum_{k=0}^{\infty}I_t^\nu(-\alpha p_{k}^\nu(0,t)+\mu p_{k}^\nu(1,t)),\ n=0,\vspace{0.1cm}\\
		p^\nu(1,0)+\sum_{k=0}^{\infty}I_t^\nu(-(\lambda+\mu)p_{k}^\nu(1,t)+\alpha p_{k}^\nu(0,t)+2\mu p_{k}^\nu(2,t)),\ n=1,\vspace{0.1cm}\\
		p^\nu(n,0)+\sum_{k=0}^{\infty}I_t^\nu(-n(\lambda+\mu)p_{k}^\nu(n,t)+(n-1)\lambda p_{k}^\nu(n-1,t)+(n+1)\mu p_{k}^\nu(n+1,t)),\ n\ge2.
	\end{cases}
\end{equation}}
By using ADM, we have
\begin{equation}\label{adm1}
	p_{0}^\nu(n,t)=p^\nu(n,0)=\begin{cases}
		1,\ n=0,\\
		0,\ n\ge1
	\end{cases}
\end{equation}
and for $k\ge1$
{\small\begin{equation}\label{adm2}
	p_{k}^\nu(n,t)=\begin{cases}
		I_t^\nu(-\alpha p_{k-1}^\nu(0,t)+\mu p_{k-1}^\nu(1,t)),\ n=0,\vspace{0.1cm}\\
		I_t^\nu(-(\lambda+\mu)p_{k-1}^\nu(1,t)+\alpha p_{k-1}^\nu(0,t)+2\mu p_{k-1}^\nu(2,t)),\ n=1,\vspace{0.1cm}\\
		I_t^\nu(-n(\lambda+\mu)p_{k-1}^\nu(n,t)+(n-1)\lambda p_{k-1}^\nu(n-1,t)+(n+1)\mu p_{k-1}^\nu(n+1,t)),\ n\ge2.
	\end{cases}
\end{equation}}
The following result provides the condition under which the series components of the state probabilities of TFLBDPwI vanish. Its proof follows by using the method of induction. A similar result  holds in the case of fractional Poisson random fields on positive plane studied in \cite{Kataria2024}. However, in this case, we deal with only one fractional integral operator, whereas their we have two fractional integral operators of different orders. Also, the system of differential equations appearing here is different from there.
\begin{proposition}\label{propadm}
	Let $p_{k}^\nu(n,t)$, $k\ge0$, $n\ge0$ be the series components as defined in (\ref{adm2}). Then, for $t\ge0$, we have $p_{k}^\nu(n,t)=0$ for all $n\ge k+1$.
\end{proposition}
\begin{proof}
	It is sufficient to show that $	p_{k}^\nu(n+k,t)=0$,
	for all $n\ge1$ and $k\ge0$. From (\ref{adm1}), we note that the result holds for $k=0$ and for any $n\geq1$. 
	
	For $k=n=1$, from (\ref{adm2}), we get
	\begin{equation*}
		\begin{split}            p_{1}^\nu(2,t)&=I_t^\nu(-2(\lambda+\mu)p_{0}^\nu(2,t)+\lambda p_{0}^\nu(1,t)+3\mu p_{0}^\nu(3,t))=I_t^\nu0=0,      
		\end{split}
	\end{equation*}
	where we have used (\ref{adm1}). So, the result holds for $n=k=1$. 
	
	Let $k=1$ and $n\ge1$. From (\ref{adm2}), we have
	\begin{equation*}
		p_{1}^\nu(n+1,t)=I_t^\nu(-(n+1)(\lambda+\mu)p_{0}^\nu(n+1,t)+n\lambda p_{0}^\nu(n,t)+(n+2)\mu p_{0}^\nu(n+2,t))=0.
	\end{equation*}
	Thus, the assertion holds for  $k=1$ and for any $n\ge1$ also. 
	
	Now, suppose  the result holds for some $k\ge2$ and for all $n\ge1$, that is, $p_k(n+k,t)=0,\ n\ge1$. Then, from (\ref{adm2}), we have
	\begin{equation*}
		\begin{split}
			p_{k+1}^\nu(n+k+1,t)&=I_t^\nu(-(n+k+1)(\lambda+\mu)p_{k}^\nu(n+k+1,t)+(n+k)\lambda p_{k}^\nu(n+k,t)\\
			&\ \ +(n+k+2)\mu p_{k}^\nu(n+k+2,t))=I_t^\nu0=0,\\           
		\end{split}
	\end{equation*}
	where we have used the induction hypothesis. This completes the proof using method of induction.
\end{proof}

Now, we explicitly derive some of the series components. From (\ref{adm2}), we have $p_{1}^\nu(0,t)=I_t^\nu(-\alpha p^\nu(0,t))$. By using (\ref{lemma1}) and (\ref{adm1}), we get $p_{1}^\nu(0,t)=-\alpha I_t^\nu(t^0)=-\alpha t^\nu/\Gamma(\nu+1)$.  

In view of Proposition \ref{propadm}, from \ref{adm2}, we have
\begin{equation}\label{1.1}
	p_{1}^\nu(1,t)=I_t^\nu(\alpha p_{0}^\nu(0,t))=\frac{\alpha t^\nu}{\Gamma(\nu+1)},
\end{equation} 
where we have used (\ref{lemma1}).
Similarly, we have
{\small\begin{align*}
	p_{2}^\nu(0,t)&=I_t^\nu(-\alpha p_{1}^\nu(0,t)+\mu p_{1}^\nu(1,t))=\frac{\alpha(\alpha+\mu)t^{2\nu}}{\Gamma(2\nu+1)},\\
	p_{2}^\nu(1,t)&=I_t^\nu(-(\lambda+\mu)p_{1}^\nu(1,t)+\alpha p_{1}^\nu(0,t)+2\mu p_{1}^\nu(2,t))=-\alpha(\alpha+\lambda+\mu)\frac{t^{2\nu}}{\Gamma(2\nu+1)},\\
	p_{3}^\nu(0,t)&=I_t^\nu(-\alpha p_{2}^\nu(0,t)+\mu p_{2}^\nu(1,t))=-\alpha(\alpha^2+2\alpha\mu+\lambda\mu+\mu^2)\frac{t^{3\nu}}{\Gamma(3\nu+1)},\\
	p_{2}^\nu(2,t)&=I_t^\nu(\lambda p_{1}^\nu(1,t))=\frac{\alpha\lambda t^{2\nu}}{\Gamma(2\nu+1)},\\
	p_{3}^\nu(1,t)&=I_t^\nu(-(\lambda+\mu)p_{2}^\nu(1,t)+\alpha p_{2}^\nu(0,t)+2\mu p_{2}^\nu(2,t))=\alpha(\alpha^2+\alpha\lambda+2\alpha\mu+2\lambda\mu+(\lambda+\mu)^2)\frac{t^{3\nu}}{\Gamma(3\nu+1)},\\
	p_{4}^\nu(0,t)&=I_t^\nu(-\alpha p_{3}^\nu(0,t)+\mu p_{3}^\nu(1,t))=\alpha(\alpha^3+3\alpha^2\mu+2\alpha\lambda\mu+3\alpha\mu^2+2\lambda\mu^2+(\lambda+\mu)^2\mu)\frac{t^{4\nu}}{\Gamma(4\nu+1)}.
\end{align*}}

Thus, for all $n\ge0$ and $k\ge0$, we can derive the series components $p^\nu_k(n,t)$ of the state probabilities of TFLBDPwI, recursively. However, for arbitrary $\alpha$, $\lambda$ and $\mu$, it is not possible here to determine the closed form of these series components. So, for the computation purpose, we only consider the case where it is possible to do so.
\begin{theorem}\label{thmadm}
	For $\alpha=\lambda=\mu$ and $n\ge0$, the series components of the state probabilities of TFLBDPwI are given by
	\begin{equation}\label{p0adm}
		p_{k}^\nu(n,t)=\begin{cases}
			\frac{(-1)^{k-n}c_{n,k}(\lambda t^\nu)^k}{\Gamma(k\nu+1)},\,k\ge n,\vspace{0.15cm}\\
			0,\ k<n,
		\end{cases}
	\end{equation}
	where $c_{0,0}=1$, $c_{0,k+1}=c_{0,k}+c_{1,k}$, $c_{1,k+1}=c_{0,k}+2(c_{1,k}+c_{2,k})$ and $c_{n,k+1}=(n-1)c_{n-1,k}+2nc_{n,k}+(n+1)c_{n+1,k}$ such that $c_{n,k}=0$ whenever $n>k$ for all $n\ge0$ and $k\ge0$.
\end{theorem}
\begin{proof}
	For $n\ge0$ and $k\ge0$ such that $n-k\ge1$, the result follows from Proposition \ref{propadm}. If $n=k=1$ then $c_{1,1}=c_{0,0}+2(c_{1,0}+c_{2,0})=1$. So, from (\ref{1.1}), the result holds for $n=k=1$. 
	
	Let us assume that (\ref{p0adm}) holds for $n\ge0$ and for some $k=m\ge2$. Then, we have the following three cases:\\	
	\textit{Case I.} For $n=0$ and $\alpha=\lambda=\mu$, from (\ref{adm2}), we have
	\begin{align*}
		p_{m+1}^\nu(0,t)&=I_t^\nu(-\lambda p_{m}^\nu(0,t)+\lambda p_{m}^\nu(1,t))\\
		&=I_t^\nu\bigg(-\lambda\frac{c_{m}(-\lambda t^\nu)^m}{\Gamma(m\nu+1)}+\lambda\frac{(-1)^{m-1}c_{1,m}(\lambda t^\nu)^{m}}{\Gamma(m\nu+1)}\bigg)\\
		&=(c_{0,m}+c_{1,m})\frac{(-\lambda t^\nu)^{m+1}}{\Gamma(\nu(m+1)+1)}=\frac{c_{0,m+1}(-\lambda t^\nu)^{m+1}}{\Gamma(\nu(m+1)+1)},
	\end{align*}
	where we have used induction hypothesis to get the second equality. Thus, the results holds for $n=0$ and $k=m+1$.\\	
	\textit{Case II.} Let $n=1$. Then, for $\alpha=\lambda=\mu$ in (\ref{adm2}), we have
	\begin{align*}
		p_{m+1}^\nu(1,t)&=I_t^\nu(-2\lambda p_{m}^\nu(1,t)+\lambda p_{m}^\nu(0,t)+2\lambda p_{m}^\nu(2,t))\\
		&=I_t^\nu\bigg(-2\lambda\frac{(-1)^{m-1}c_{1,m}(\lambda t^\nu)^{m}}{\Gamma(m\nu+1)}+\lambda\frac{c_{0,m}(-\lambda t^\nu)^m}{\Gamma(m\nu+1)}+2\lambda\frac{(-1)^{m-2}c_{2,m}(\lambda t^\nu)^m}{\Gamma(m\nu+1)}\bigg)\\
		&=(2c_{1,m}+c_{0,m}+2c_{2,m})\frac{(-1)^m(\lambda t^\nu)^{m+1}}{\Gamma(\nu(m+1)+1)}=\frac{c_{1,m+1}(-1)^m(\lambda t^\nu)^{m+1}}{\Gamma(\nu(m+1)+1)},
	\end{align*}
	where again the second equality follows from induction hypothesis. So, the result holds for $n=1$ and $k=m+1$.\\
	\textit{Case III.} If $n\ge2$ and $n-k\ge1$ then the result follows from Proposition \ref{propadm}. Now, suppose $n-m=1$. Then, $p_{m}^\nu(n,t)=p_{m}^\nu(n+1,t)=0$ and from (\ref{adm2}), we get
	\begin{equation*}
		p_{m+1}^\nu(n,t)=I_t^\nu(m\lambda p_{m}^\nu(m,t))=I_t^\nu\bigg(m\lambda\frac{c_{m,m}(\lambda t^\nu)^m}{\Gamma(m\nu+1)}\bigg)=\frac{c_{m+1,m+1}(\lambda t^\nu)^{m+1}}{\Gamma((m+1)\nu+1)},
	\end{equation*}
	where we have used $c_{m+1,m+1}=2(m+1)c_{m+1,m}+mc_{m,m}+(m+2)c_{m+2,m}=mc_{m,m}$ since $c_{m+2,m}=c_{m+1,m}=0$.
	
	If $n-m=0$ then $p_{m}^\nu(n+1,t)=0$. So,
	\begin{align*}
		p_{m+1}^\nu(n,t)&=I_t^\nu\bigg(-2m\lambda\frac{c_{m,m}(\lambda t^\nu)^m}{\Gamma(m\nu+1)}-(m-1)\lambda\frac{c_{m-1,m}(\lambda t^\nu)^m}{\Gamma(\nu m+1)}\bigg)\\
		&=-(2mc_{m,m}+(m-1)c_{m-1,m})\frac{(\lambda t^\nu)^{m+1}}{\Gamma((m+1)\nu+1)}=\frac{-c_{m,m+1}(\lambda t^\nu)^{m+1}}{\Gamma((m+1)\nu+1)},
	\end{align*}
	where we have used $c_{m+1,m}=0$ to get $c_{m,m+1}=2mc_{m,m}+(m-1)c_{m-1,m}$.
	
	Let $n-m<0$. Then, from (\ref{adm2}), we get
	\begin{align*}
		p_{m+1}^{\nu}(n,t)&=I_t^\nu(2n\lambda p_{m}^\nu(n,t)+(n-1)\lambda p_{m}^\nu(n-1,t)+(n+1)\lambda p_{m}^\nu(n+1,t))\\
		&=I_t^\nu\bigg(-2n\lambda \frac{(-1)^{m-n}c_{n,m}(\lambda t^\nu)^m}{\Gamma(m\nu+1)}+(n-1)\lambda\frac{(-1)^{m-n+1}c_{n-1,m}(\lambda t^\nu)^{m}}{\Gamma(m\nu+1)}\\
		&\ \ +(n+1)\lambda\frac{(-1)^{m-n-1}c_{n+1,m}(\lambda t^\nu)^{m}}{\Gamma(m\nu+1)}\bigg)\\
		&=(2nc_{n,m}+(n-1)c_{n-1,m}+(n+1)c_{n+1,m})\frac{(-1)^{m+1-n}(\lambda t^\nu)^{m+1}}{\Gamma((m+1)\nu+1)}.
	\end{align*}
	Thus, (\ref{p0adm}) holds for $n\ge2$ and $k=m+1$ too. This completes the proof by using the method of induction.
\end{proof}

On summing (\ref{p0adm}) over $k=0,1,2,\dots$, we get the following result.
\begin{theorem}
	For $\alpha=\lambda=\mu$,  state probabilities of the TFLBDPwI are given by
	\begin{equation}\label{tflbdpsp}
		p^\nu(n,t)=\sum_{k=n}^{\infty}\frac{(-1)^{k-n}c_{n,k}(\lambda t^\nu)^k}{\Gamma(k\nu+1)},\ n\ge0,
	\end{equation}
	where $p^\nu(0,t)$ is called the extinction probability.
\end{theorem}
\begin{remark}
	The regularity of state probabilities (\ref{tflbdpsp}) can be checked as follows:
	\begin{align*}
		\sum_{n=0}^{\infty}p^\nu(n,t)&=\sum_{n=0}^{\infty}\sum_{k=n}^{\infty}\frac{(-1)^{k-n}c_{n,k}(\lambda t^\nu)^k}{\Gamma(k\nu+1)}\\
		&=\sum_{k=0}^{\infty}\frac{c_{0,k}(-\lambda t^\nu)^k}{\Gamma(k\nu+1)}+\sum_{k=1}^{\infty}\frac{(-1)^{k-1}c_{1,k}(\lambda t^\nu)^k}{\Gamma(k\nu+1)}+\sum_{n=2}^{\infty}\sum_{k=n}^{\infty}\frac{(-1)^{k-n}c_{n,k}(\lambda t^\nu)^k}{\Gamma(k\nu+1)}\\
		&=\sum_{k=0}^{\infty}\frac{c_{0,k}(-\lambda t^\nu)^k}{\Gamma(k\nu+1)}+\sum_{k=1}^{\infty}\frac{(-1)^{k-1}c_{1,k}(\lambda t^\nu)^k}{\Gamma(k\nu+1)}+\sum_{k=2}^{\infty}\frac{(\lambda t^\nu)^k}{\Gamma(k\nu+1)}\sum_{n=2}^{k}(-1)^{k-n}c_{n,k}\\
		&=1+\sum_{k=2}^{\infty}\frac{(\lambda t^\nu)^k}{\Gamma(k\nu+1)}\sum_{n=0}^{k}(-1)^{k-n}c_{n,k}=1,
	\end{align*}
	where we have used $c_{0,0}=c_{0,1}=c_{1,1}=1$ and the recurrence relation among the coefficients given in Theorem \ref{thmadm} to show
	\begin{align*}
		\sum_{n=0}^{k}(-1)^{k-n}c_{n,k}
		&=(-1)^kc_{0,k}-(-1)^kc_{1,k}+\sum_{n=2}^{k-1}(-1)^{k-n}2nc_{n,k-1}\\
		&\ \ +\sum_{n=2}^{k}(-1)^{k-n}(n-1)c_{n-1,k-1}+\sum_{n=2}^{k-2}(-1)^{k-n}(n+1)c_{n+1,k-1}\\
		&=(-1)^kc_{0,k}-(-1)^kc_{1,k}+(-1)^{k-2}4c_{2,k-1}-\sum_{n=2}^{k-2}(-1)^{k-n}2(n+1)c_{n+1,k-1}\\
		&\ \ +(-1)^kc_{1,k-1}-(-1)^k2c_{2,k-1}+2\sum_{n=2}^{k-2}(-1)^{n-k}(n+1)c_{n+1,k-1}\\
		&=(-1)^k(c_{0,k-1}+c_{1,k-1})-(-1)^k(c_{0,k-1}+2c_{1,k-1}+2c_{2,k-1})\\
		&\ \ +(-1)^kc_{1,k-1}+(-1)^k2c_{2,k-1}=0,\ k\ge2.
	\end{align*}
\end{remark}

Note that for $\alpha=\lambda=\mu$, the pgf of TFLBDPwI solves the following fractional differential equation:
\begin{equation}\label{frequpgf}
	\frac{\partial^\nu}{\partial t^\nu}G^\nu(u,t)=\lambda(u-1)^2\frac{\partial}{\partial u}G^\nu(u,t)+\lambda(u-1)p^\nu(0,t),
\end{equation}
with initial condition $G^\nu(u,0)=1$. So, on taking the derivative with respect to $u$ on both sides of (\ref{frequpgf}) and substituting $u=1$, we get the governing differential equation for its mean $\mathbb{E}N^\nu(t)=(\partial/\partial u)G^\nu(u,t)|_{u=1}$ as follows:
\begin{equation}\label{remean}
	\frac{\mathrm{d}^\nu}{\mathrm{d}t^\nu}\mathbb{E}N^\nu(t)=\lambda p^\nu(0,t)=\lambda\sum_{k=0}^{\infty}\frac{c_{0,k}(-\lambda t^\nu)^k}{\Gamma(k\nu+1)},
\end{equation}
with $\mathbb{E}N^\nu(0)=0$.
On taking the Laplace transform on both sides of (\ref{remean}) and using (\ref{frderlap}), we get
\begin{equation*}
	z^\nu\int_{0}^{\infty}e^{-zt}\mathbb{E}N^\nu(t)\,\mathrm{d}t-z^{\nu-1}\mathbb{E}N^\nu(0)=\lambda\sum_{k=0}^{\infty}\frac{c_{0,k}(-\lambda)^k}{z^{k\nu+1}}.
\end{equation*}
So,
\begin{equation}\label{remeanlap}
	\int_{0}^{\infty}e^{-zt}\mathbb{E}N^\nu(t)\,\mathrm{d}t=\lambda\sum_{k=0}^{\infty}\frac{c_{0,k}(-\lambda)^k}{z^{(k+1)\nu+1}},\ z>0,
\end{equation}
whose inversion yields 
\begin{equation*}
	\mathbb{E}N^\nu(t)=-\sum_{k=0}^{\infty}\frac{c_{0,k}(-\lambda t^\nu)^{k+1}}{\Gamma((k+1)\nu+1)},\ t\ge0.
\end{equation*}
Now, on taking the derivative twice with respect to $u$ on both sides of (\ref{frequpgf}) and substituting $u=1$, we have
\begin{equation*}
	\frac{\mathrm{d}^\nu}{\mathrm{d}t^\nu}\mathbb{E}N^\nu(t)(N^\nu(t)-1)=2\lambda\mathbb{E}N^\nu(t),
\end{equation*}
with $\mathbb{E}N^\nu(0)(N^\nu(0)-1)=0$, where $\mathbb{E}N^\nu(t)(N^\nu(t)-1)=(\partial^2/\partial u^2)G^\nu(u,t)|_{u=1}$ is the second factorial moment of TFLBDPwI. By using (\ref{frderlap}) and (\ref{remeanlap}), its Laplace transform is given by
\begin{equation*}
	\int_{0}^{\infty}e^{-zt}\mathbb{E}N^\nu(t)(N^\nu(t)-1)\,\mathrm{d}t=2\lambda^2\sum_{k=0}^{\infty}\frac{c_{0,k}(-\lambda)^k}{z^{(k+2)\nu+1}},\ z>0,
\end{equation*}
whose inversion yields
\begin{equation*}
	\mathbb{E}N^\nu(t)(N^\nu(t)-1)=2\sum_{k=0}^{\infty}\frac{c_{0,k}(-\lambda t^\nu)^{k+2}}{\Gamma((k+2)\nu+1)},\ t\ge0.
\end{equation*}
Thus, in the case of $\alpha=\lambda=\mu$, the variance of TFLBDPwI is
\begin{equation*}
	\mathbb{V}\mathrm{ar}N^\nu(t)=\sum_{k=0}^{\infty}c_{0,k}(-\lambda t^\nu)^{k+1}\bigg(-\frac{2\lambda t^\nu}{\Gamma((k+2)\nu+1)}-\frac{1}{\Gamma((k+1)\nu+1)}\bigg)-(\mathbb{E}N^\nu(t))^2,\ t\ge0.
\end{equation*}
\begin{remark}
	Alternatively, for $\alpha=\lambda=\mu$, the mean value of TFLBDPwI can be derived as follows:
	\begin{align}
		\mathbb{E}N^\nu(t)&=\sum_{n=0}^{\infty}n\sum_{k=n}^{\infty}\frac{(-1)^{k-n}c_{n,k}(\lambda t^\nu)^k}{\Gamma(k\nu+1)}\nonumber\\
		&=\sum_{k=1}^{\infty}\frac{(\lambda t^\nu)^k}{\Gamma(k\nu+1)}\sum_{n=1}^{k}(-1)^{k-n}nc_{n,k}\nonumber\\
		&=\sum_{k=1}^{\infty}\frac{(\lambda t^\nu)^k}{\Gamma(k\nu+1)}\bigg(-(-1)^kc_{1,k}++\sum_{n=2}^{k-1}(-1)^{k-n}2n^2c_{n,k-1}\nonumber\\
		&\ \ +\sum_{n=2}^{k}(-1)^{k-n}n(n-1)c_{n-1,k-1}+\sum_{n=2}^{k-2}(-1)^{k-n}n(n+1)c_{n+1,k-1}\bigg),\label{meanadm}
	\end{align}
	where we have used the recurrence relation of the coefficients mentioned in Theorem \ref{thmadm}. Now, on substituting $n(n-1)=(n-1)^2+(n-1)$ and $n(n+1)=(n+1)^2-(n+1)$ in (\ref{meanadm}), we get
	\begin{align*}
		\mathbb{E}N^\nu(t)&=\sum_{k=1}^{\infty}\frac{(\lambda t^\nu)^k}{\Gamma(k\nu+1)}\bigg(-(-1)^kc_{1,k}+\sum_{n=2}^{k-1}(-1)^{k-n}2n^2c_{n,k-1}\\
		&\ \ +\sum_{n=2}^{k}(-1)^{k-n}(n-1)^2c_{n-1,k-1}+\sum_{n=2}^{k}(-1)^{k-n}(n-1)c_{n-1,k-1}\\
		&\ \ +\sum_{n=2}^{k-2}(-1)^{k-n}(n+1)^2c_{n+1,k-1}-\sum_{n=2}^{k-2}(-1)^{k-n}(n+1)c_{n+1,k-1}\bigg)\\
		&=\sum_{k=1}^{\infty}\frac{(\lambda t^\nu)^k}{\Gamma(k\nu+1)}\bigg(-(-1)^kc_{1,k}+\sum_{n=2}^{k-1}(-1)^{k-n}2n^2c_{n,k-1}\\
		&\ \ -\sum_{n=1}^{k-1}(-1)^{k-n}n^2c_{n,k-1}-\sum_{n=1}^{k-1}(-1)^{k-n}nc_{n,k-1}\\
		&\ \ -\sum_{n=3}^{k-1}(-1)^{k-n}n^2c_{n,k-1}+\sum_{n=3}^{k-1}(-1)^{k-n}nc_{n,k-1}\bigg)\\
		&=\sum_{k=1}^{\infty}\frac{(\lambda t^\nu)^k}{\Gamma(k\nu+1)}(-(-1)^kc_{0,k-1}-(-1)^k2(c_{1,k-1}+c_{2,k-1})\\
		&\ \ +(-1)^{k}8c_{2,k-1}+(-1)^kc_{1,k-1}-(-1)^k4c_{2,k-1}+(-1)^kc_{1,k-1}-(-1)^k2c_{2,k-1})\\
		&=\sum_{k=1}^{\infty}\frac{-c_{0,k-1}(-\lambda t^\nu)^k}{\Gamma(k\nu+1)}.
	\end{align*}
	\end{remark}
	\begin{remark}
		Note that for each $n\ge0$ the coefficients $\{c_{n,k}\}_{k\ge0}$ of the series components is a very rapidly growing sequence of non-negative integers. These coefficients can be calculated by using the given recurrence relation. Moreover, they appear in a triangular array structure which is given as follows:
		\begin{equation*}
			\begin{array}{c|ccccccccccc}
				k& c_{0,k} & c_{1,k} &c_{2,k}&c_{3,k}&c_{4,k}&c_{5,k}&c_{6,k}&c_{7,k}&c_{8,k}&\cdots\\ \hline
				0 & 1 & 0 & 0&0&0&0&0&0&0&\dots \\
				1 & 1 & 1 & 0&0&0&0 &0&0&0&\dots\\
				2 & 2 & 3 & 1&0&0&0&0&0&0&\dots\\
				3&5&10&7&2&0&0&0&0&0&\dots\\
				4&15&39&44&26&6&0&0&0&0&\dots\\
				5&54&181&293&268&126&24&0&0&0&\dots\\
				6&235&1002&2157&2698&1932&744&120&0&0&\dots\\
				7&1237&6553&17724&28230&27270&15888&5160&720&0&\dots\\
				8&7790&49791&162139&313908&382290&298920&146400&41040&5040&\dots\\
				\vdots&\vdots&\vdots&\vdots&\vdots&\vdots&\vdots&\vdots&\vdots&\vdots
			\end{array}
		\end{equation*} 
	For $\lambda = 1$, $\eta = 1$ and $t \in \{0.4, 0.6, 0.8, 0.9, 1\}$, the growth of series components for $p^\nu(0,t)$, $p^\nu(1,t)$, $p^\nu(2,t)$, and $p^\nu(3,t)$ is illustrated in Figure \ref{fig1}. It can be observed that for $n = 0$, as $k$ increases, the series components approach zero, with the convergence rate decreasing as $t$ increases. Moreover, for higher $n$ values, the deviation from zero becomes larger as $t$ increases. Therefore, we conclude that to get a better approximation for the state probabilities for larger values of $n$ and $t$, we require a higher number of series components.
\end{remark}	
	\begin{figure}
		\includegraphics[width=16cm]{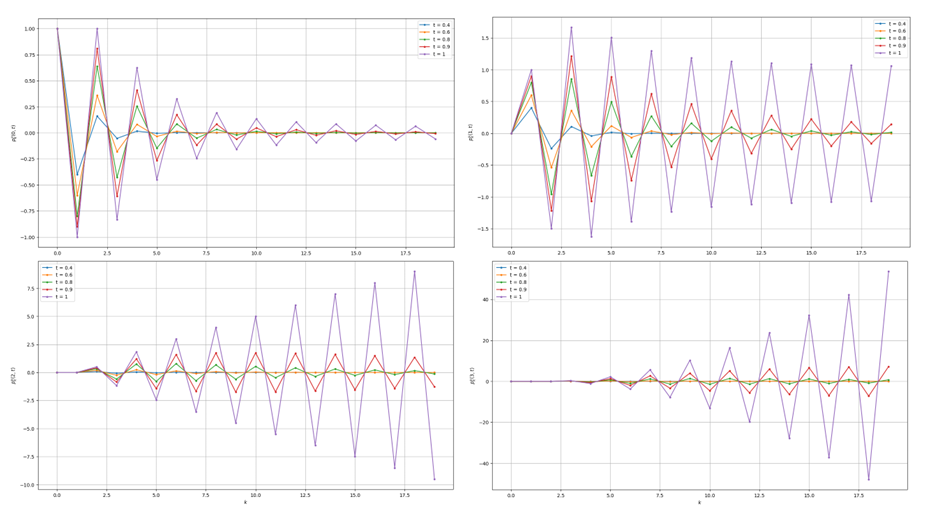}
		\caption{\small Growth of series components of the state probabilities for $\lambda=1$ and $\eta=1$.}\label{fig1}
	\end{figure}
	
	\section{Some particular cases of TFLBDPwI}
	In this section, we study the time-changed variants of some particular cases of LBDPwI, such as, the linear birth process with immigration and linear death immigration process. 
\subsection{Time-fractional linear birth process with immigration}
If we allow $\mu=0$ in LBDPwI then its time-changed version \textit{viz} TFLBDPwI reduces to a pure linear birth process with positive probability of immigration at state $n=0$. We call this process as the time-fractional linear birth process with immigration (TFLBPwI) and denote it by $\{\bar{N}^\nu(t)\}_{t\ge0}$, $0>\nu\leq1$. Its state probabilities $\bar{p}^\nu(n,t)=\mathrm{Pr}\{\bar{N}^\nu(t)=n\},\ n\ge0$ solve
\begin{equation}\label{fbpwiequ}
	\frac{\mathrm{d}^\nu }{\mathrm{d}t^\nu}\bar{p}^\nu(n,t)=\begin{cases}
	-\alpha \bar{p}^\nu(0,t),\ n=0,\vspace{0.1cm}\\
	-\lambda \bar{p}^\nu(1,t)+\alpha \bar{p}^\nu(0,t),\ n=1,\vspace{0.1cm}\\
	-n\lambda \bar{p}^\nu(n,t)+(n-1)\lambda \bar{p}^\nu(n-1,t),\ n\ge2.
\end{cases}
\end{equation}

Similar to the case of TFLBDPwI, we can use the ADM to obtain the state probabilities of TFLBPwI. Let $\bar{p}^\nu(n,t)=\sum_{k=0}^{\infty}\bar{p}_{k}^\nu(n,t)$ for all $n\ge0$. Then, $\bar{p}_k^\nu(n,t)=0$ whenever $n-k\ge1$. Moreover, if we allow $\mu=0$ then the non-zero series components $p_k^\nu(n,t)$ of TFLBDPwI will reduce to $\bar{p}_k^\nu (n,t)$ for all $n\ge0$ which are given as follows:
{\footnotesize\begin{align*}
	&\bar{p}^\nu_0(0,t)=1,\ \ \bar{p}_1^\nu(0,t)=-\frac{\alpha t^\nu}{\Gamma(\nu+1)},\ \ \bar{p}_1^\nu(1,t)=\frac{\alpha t^\nu}{\Gamma(\nu+1)},\ \ \bar{p}_2^\nu(0,t)=\frac{(\alpha t^\nu)^2}{\Gamma(2\nu+1)},\\
	&\bar{p}_2^\nu(1,t)=-\frac{\alpha(\alpha+\lambda)t^{2\nu}}{\Gamma(2\nu+1)},\ \ \bar{p}_3^\nu(0,t)=-\frac{(\alpha t^\nu)^3}{\Gamma(3\nu+1)},\ \ \bar{p}_2^\nu(2,t)=\frac{\alpha\lambda t^{2\nu}}{\Gamma(2\nu+1)},\ \ \bar{p}_3^\nu(1,t)=\frac{\alpha(\alpha^2+\alpha\lambda+\lambda^2)t^{3\nu}}{\Gamma(3\nu+1)},\dots.
\end{align*}}
\begin{theorem}
	For $t\ge0$, the series components of the state probabilities of TFLBPwI are given by
	\begin{equation}\label{tflbadm}
		\bar{p}_k^\nu(n,t)=\begin{cases}
			\frac{(-\alpha t^\nu)^k}{\Gamma(k\nu+1)},\ n=0,\ k\ge0,\vspace{0.1cm}\\
			\sum_{r=1}^{n}(-1)^{r+k}\binom{n-1}{r-1}\frac{\alpha((r\lambda)^{k}-\alpha^{k})t^{k\nu}}{(r\lambda-\alpha)\Gamma(k\nu+1)},\ n\ge1,\ k\ge n,\vspace{0.1cm}\\
			0,\ n-k\ge1.
		\end{cases}
	\end{equation}
\end{theorem}
\begin{proof}
	The proof follows along the similar lines to that of Theorem \ref{thmadm}. Hence, it is omitted.
\end{proof}
\begin{theorem}
	The state probabilities of TFLBPwI are 
	\begin{equation}\label{tflbppmf}
		\bar{p}^\nu(n,t)=\begin{cases}
			E_{\nu,1}(-t^\nu\alpha),\ n=0,\\
			\sum_{r=1}^{n}(-1)^r\binom{n-1}{r-1}\frac{\alpha}{r\lambda-\alpha}(E_{\nu,1}(-t^\nu r\lambda)-E_{\nu,1}(-t^\nu\alpha)),\ n\ge1,
		\end{cases}
	\end{equation}
	where $E_{\nu,1}(\cdot)$ is the Mittag-Leffler function defined as (see \cite{Kilbas2006})
	\begin{equation}\label{mittaglfr}
		E_{\nu,1}(x)=\sum_{k=0}^{\infty}\frac{x^k}{\Gamma(k\nu+1)},\ x\in\mathbb{R}.
	\end{equation}
\end{theorem}
\begin{proof}
	For $t\ge0$, on summing (\ref{tflbadm}) over the range of $k$, we get
	\begin{equation*}
		\bar{p}^\nu(0,t)=\sum_{k=0}^{\infty}\frac{(-\alpha t^\nu)^k}{\Gamma(k\nu+1)}=E_{\nu,1}(-t^\nu\alpha)
	\end{equation*}
	and
	\begin{align}
		\bar{p}^\nu(n,t)&=\sum_{k=n}^{\infty}\sum_{r=1}^{n}(-1)^{r+k}\binom{n-1}{r-1}\frac{\alpha((r\lambda)^{k}-\alpha^{k})t^{k\nu}}{(r\lambda-\alpha)\Gamma(k\nu+1)}\nonumber\\
		&=\sum_{r=1}^{n}(-1)^r\binom{n-1}{r-1}\frac{\alpha}{r\lambda-\alpha}\left(\sum_{k=1}^{\infty}\left(\frac{(-r\lambda t^\nu)^k}{\Gamma(k\nu+1)}-\frac{(-\alpha t^\nu)^k}{\Gamma(k\nu+1)}\right)-\sum_{k=1}^{n-1}(-1)^k\frac{((r\lambda)^{k}-\alpha^{k})t^{k\nu}}{\Gamma(k\nu+1)}\right)\nonumber\\
		&=\sum_{r=1}^{n}(-1)^r\binom{n-1}{r-1}\frac{\alpha}{r\lambda-\alpha}\left(\sum_{k=0}^{\infty}\frac{(-r\lambda t^\nu)^k}{\Gamma(k\nu+1)}-\sum_{k=0}^{\infty}\frac{(-\alpha t^\nu)^k}{\Gamma(k\nu+1)}\right)\nonumber\\
		&\ \ -\sum_{r=1}^{n}(-1)^r\binom{n-1}{r-1}\frac{\alpha}{r\lambda-\alpha}\sum_{k=1}^{n-1}(-1)^k\frac{((r\lambda)^{k}-\alpha^{k})t^{k\nu}}{(r\lambda-\alpha)\Gamma(k\nu+1)}\nonumber\\
		&=\sum_{r=1}^{n}(-1)^r\binom{n-1}{r-1}\frac{\alpha}{r\lambda-\alpha}\left(\sum_{k=0}^{\infty}\frac{(-r\lambda t^\nu)^k}{\Gamma(k\nu+1)}-\sum_{k=0}^{\infty}\frac{(-\alpha t^\nu)^k}{\Gamma(k\nu+1)}\right)\nonumber\\
		&\ \ -\alpha\sum_{k=1}^{n-1}\frac{(-t^\nu)^k}{\Gamma(k\nu+1)}\sum_{j=0}^{k-1}\lambda^j\alpha^{k-1-j}\sum_{r=1}^{n}(-1)^{r}\binom{n-1}{r-1}r^j\nonumber\\
		&=\sum_{r=1}^{n}(-1)^r\binom{n-1}{r-1}\frac{\alpha}{r\lambda-\alpha}\left(E_{\nu,1}(-r\lambda t^\nu)-E_{\nu,1}(-\alpha t^\nu)\right)\nonumber
	\end{align}
	where we have used $a^k-b^k=(a-b)\sum_{j=0}^{k-1}a^{k-1-j}b^j$, $k\ge1$ and
\begin{equation}\label{0eq}
	\sum_{r=1}^{n}(-1)^{r}\binom{n-1}{r-1}r^j=0
\end{equation}
for all $0\leq j<n-1$. This establishes the required result. To show that (\ref{0eq}) holds true, we proceed as follows:
\begin{equation*}
	(1-e^x)^n=\sum_{r=0}^{n}(-1)^r\binom{n}{r}e^{rx},\ x\in\mathbb{R},\ n\ge2,
\end{equation*}
which on taking derivative with respect to $x$ yields
\begin{equation}\label{0eq1}
	-e^x(1-e^x)^{n-1}=\sum_{r=1}^{n}(-1)^r\binom{n-1}{r-1}e^{rx},\ x\in\mathbb{R}.
\end{equation}
On taking the derivative of order $j<n-1$ on both sides of (\ref{0eq1}) and substituting $x=0$, we get (\ref{0eq}).
\end{proof}

Alternatively, we can use the Laplace method to obtain the state probabilities (\ref{tflbppmf}). The following result gives the explicit form of the Laplace transform of the state probabilities of TFLBPwI.
\begin{theorem}\label{thm3.3.}
	The Laplace transforms of the state probabilities of TFLBPwI are given by
	\begin{equation}\label{p01lap}
		\tilde{\bar{p}}^\nu(n,z)=\begin{cases}
			\displaystyle \frac{z^{\nu-1}}{z^\nu+\alpha},\ n=0,\vspace{0.1cm}\\
			\displaystyle\sum_{k=1}^{n}\binom{n-1}{k-1}\frac{(-1)^{k}\alpha}{k\lambda-\alpha}\left(\frac{z^{\nu-1}}{z^\nu+k\lambda}-\frac{z^{\nu-1}}{z^\nu+\alpha}\right),\ n\ge1.
		\end{cases}
	\end{equation}
\end{theorem}
\begin{proof}
	On taking $n=0$ in (\ref{fbpwiequ}), we get
	$
		({\mathrm{d}^\nu}/{\mathrm{d}t^\nu})\bar{p}^\nu(0,t)=-\alpha \bar{p}^\nu(0,t) 
	$
	with initial condition $\bar{p}^\nu(0,0)=1.$
	On using (\ref{frderlap}), its Laplace transform is given by $\tilde{\bar{p}}^\nu(0,z)=z^{\nu-1}/(z^\nu+\alpha)$. For $n=1$ in (\ref{fbpwiequ}), we get
	\begin{equation}\label{p01}
		\bar{p}^\nu(1,t)=-\lambda \bar{p}^\nu(1,t)+\alpha \bar{p}^\nu(0,t)
	\end{equation}
	with $\bar{p}^\nu(1,0)=0$.
	On taking Laplace transform on both sides of (\ref{p01}) and using (\ref{frderlap}), we get
	\begin{equation*}
		\tilde{\bar{p}}^\nu(1,z)=\frac{\alpha}{\lambda-\alpha}\left(\frac{z^{\nu-1}}{z^\nu+\alpha}-\frac{z^{\nu-1}}{z^\nu+\lambda}\right).
	\end{equation*}
	So, (\ref{p01lap}) holds true for $n=0,1$.
	Let us assume that it holds true for some $n=m-1$, where $m\ge2$, that is,
	\begin{equation}\label{indhyp1}
		\tilde{\bar{p}}^\nu(m-1,z)=\sum_{k=1}^{m-1}\binom{m-2}{k-1}\frac{(-1)^{k}\alpha}{k\lambda-\alpha}\left(\frac{z^{\nu-1}}{z^\nu+k\lambda}-\frac{z^{\nu-1}}{z^\nu+\alpha}\right).
	\end{equation}
	For $n=m$, from (\ref{fbpwiequ}), we have 
	\begin{equation}\label{cr2}
		\frac{\mathrm{d}^\nu }{\mathrm{d}t^\nu}p^\nu(m,t)=-m\lambda p^\nu(m,t)+(m-1)\lambda p^\nu(m-1,t),\ p^\nu(m,0)=0.
	\end{equation}
	On taking the Laplace transform on both sides (\ref{cr2}) and using (\ref{indhyp1}), we get
	\begin{align*}
		\tilde{\bar{p}}^\nu(m,z)&=(m-1)\lambda\sum_{k=1}^{m-1}\binom{m-2}{k-1}\frac{(-1)^{k}\alpha}{k\lambda-\alpha}\left(\frac{z^{\nu-1}}{(z^\nu+k\lambda)(z^\nu+m\lambda)}-\frac{z^{\nu-1}}{(z^\nu+\alpha)(z^\nu+m\lambda)}\right)\\
		&=(m-1)\lambda\sum_{k=1}^{m-1}\binom{m-2}{k-1}\frac{(-1)^{k}\alpha}{k\lambda-\alpha}\frac{1}{(m-k)\lambda}\left(\frac{z^{\nu-1}}{z^\nu+k\lambda}-\frac{z^{\nu-1}}{z^\nu+m\lambda}\right)\nonumber\\
		&\ \ -(m-1)\lambda\sum_{k=1}^{m-1}\binom{m-2}{k-1}\frac{(-1)^{k}\alpha}{k\lambda-\alpha}\frac{1}{m\lambda-\alpha}\left(\frac{z^{\nu-1}}{z^\nu+\alpha}-\frac{z^{\nu-1}}{z^\nu+m\lambda}\right)\nonumber\\
		&=\sum_{k=1}^{m-1}\binom{m-1}{k-1}\frac{(-1)^{k}\alpha}{k\lambda-\alpha}\frac{z^{\nu-1}}{z^\nu+k\lambda}-\sum_{k=1}^{m-1}\binom{m-1}{k-1}\frac{(-1)^{k}\alpha}{k\lambda-\alpha}\frac{(m-k)\lambda}{m\lambda-\alpha}\frac{z^{\nu-1}}{z^\nu+\alpha}\\
		&\ \ +\frac{\alpha}{(m\lambda-\alpha)}\frac{z^{\nu-1}}{z^\nu+m\lambda}\sum_{k=1}^{m-1}\binom{m-1}{k-1}{(-1)^{k-1}}\\
		&=\sum_{k=1}^{m-1}\binom{m-1}{k-1}\frac{(-1)^{k}\alpha}{k\lambda-\alpha}\frac{z^{\nu-1}}{z^\nu+k\lambda}-\sum_{k=1}^{m-1}\binom{m-1}{k-1}\frac{(-1)^{k}\alpha}{k\lambda-\alpha}\frac{z^{\nu-1}}{z^\nu+\alpha}\\
		&\ \ -\frac{\alpha}{m\lambda-\alpha}\frac{z^{\nu-1}}{z^\nu+\alpha}\sum_{k=1}^{m-1}\binom{m-1}{k-1}(-1)^{k-1}-\frac{\alpha(-1)^{m-1}}{m\lambda-\alpha}\frac{z^{\nu-1}}{z^\nu+m\lambda}\\
		&=\sum_{k=1}^{m}\binom{m-1}{k-1}\frac{(-1)^{k}\alpha}{k\lambda-\alpha}\frac{z^{\nu-1}}{z^\nu+k\lambda}-\sum_{k=1}^{m}\binom{m-1}{k-1}\frac{(-1)^{k}\alpha}{k\lambda-\alpha}\frac{z^{\nu-1}}{z^\nu+\alpha}.
		\end{align*}
	Thus, the proof is complete by using the method of induction.
\end{proof}
	\begin{remark}\label{thm22}
		On taking the inverse Laplace transform of (\ref{p01lap}) and using the following result (see \cite{Kilbas2006})
		\begin{equation}\label{Mittag-lap}
			\int_{0}^{\infty}e^{-zx}E_{\nu,1}(cx^\nu)\,\mathrm{d}x=\frac{z^{\nu-1}}{z^\nu-c},\ z>0,\, c\in\mathbb{R}, 
		\end{equation}
		 we get the state probabilities of TFLBPwI as follows:
		\begin{equation}\label{00p0}
			\bar{p}^\nu(n,t)=\begin{cases}
				E_{\nu,1}(-t^\nu\alpha),\ n=0,\vspace{0.1cm}\\
				\displaystyle\sum_{k=1}^{n}\binom{n-1}{k-1}\frac{(-1)^{k}\alpha}{k\lambda-\alpha}(E_{\nu,1}(-kt^\nu\lambda)-E_{\nu,1}(-t^\nu\alpha)),\ n\ge1,
			\end{cases}
		\end{equation}
		which agree with (\ref{tflbppmf}).
		
		In particular, for $\nu=1$, the TFLBPwI reduces to the linear birth process with immigration. Its probability mass function is 
		\begin{equation*}
			\bar{p}(n,t)=\begin{cases}
				e^{-\alpha t},\ n=0,\vspace{0.1cm}\\
				\displaystyle\sum_{k=1}^{n}\binom{n-1}{k-1}\frac{(-1)^{k}\alpha}{k\lambda-\alpha}(e^{-kt\lambda}-e^{-t\alpha}),\ n\ge1.
			\end{cases}
		\end{equation*}
	\end{remark}

\begin{remark}
	If $T^\nu$ denotes the waiting time of  the TFLBPwI in state $n=0$ then it follows Mittag-Leffler distribution with parameter $\alpha$, that is, $\mathrm{Pr}\{T^\nu> s\}=E_{\nu,1}(-s^\nu\alpha)$, $s>0$.  
\end{remark}

\begin{remark}
	If the rate of immigration $\alpha$ is sufficiently small enough such that it just trigger the process, that is, $\alpha^l=o(\alpha)$ for all $l\ge2$ and $k\lambda-\alpha\approx k\lambda$, where $o(\alpha)/\alpha\to 0$ as $\alpha\to 0$, then
	$$
		E_{\nu,1}(-t^\nu\alpha)= 1-{t^\nu\alpha}/{\Gamma(1+\nu)}+o(\alpha).
$$
	 So, from (\ref{00p0}), the state probabilities of TFLBPwI reduces to
	$
		\bar{p}^\nu(0,t)=1-{t^\nu\alpha}/{\Gamma(1+\nu)}+o(\alpha),\ t\ge0
	$
	and
	\begin{align*}
		\bar{p}^\nu(n,t)&=\frac{\alpha}{n\lambda}\left(\sum_{k=1}^{n}\binom{n}{k}(-1)^kE_{\nu,1}(-kt^\nu\lambda)-\left(1-\frac{t^\nu\alpha}{\Gamma(1+\nu)}\right)\sum_{k=1}^{n}\binom{n}{k}(-1)^k\right)+o(\alpha)\\
		&=\frac{\alpha}{n\lambda}\left(\sum_{k=1}^{n}\binom{n}{k}(-1)^kE_{\nu,1}(-kt^\nu\lambda)+1\right)-\frac{t^\nu\alpha^2}{n\lambda\Gamma(1+\nu)}+o(\alpha)\\
		&=\frac{\alpha}{n\lambda}\sum_{k=0}^{n}\binom{n}{k}(-1)^kE_{\nu,1}(-kt^\nu\lambda)+o(\alpha),\ n\ge1.
	\end{align*}
	For $\nu=1$, we get
	$
		\bar{p}(0,t)=1-\alpha t+o(\alpha)\ \text{and}\ \bar{p}(n,t)={\alpha(1-e^{-t\lambda})^n}/{n\lambda}+o(\alpha),\ n\ge1.
$
\end{remark}

\begin{theorem}
	The state probabilities (\ref{tflbppmf}) satisfy the regularity condition, that is,
	$
		\sum_{n=0}^{\infty}\bar{p}^\nu(n,t)$ $=1.
$
\end{theorem}
\begin{proof}
	From (\ref{00p0}), we have
	\begin{equation}\label{regular1}
\sum_{n=0}^{\infty}\bar{p}^\nu(n,t)=E_{\nu,1}(-t^\nu\alpha)+\sum_{n=1}^{\infty}\sum_{k=1}^{n}\binom{n-1}{k-1}\frac{(-1)^{k}\alpha}{k\lambda-\alpha}(E_{\nu,1}(-kt^\nu\lambda)-E_{\nu,1}(-t^\nu\alpha)).
	\end{equation}
	On taking the Laplace transform on both sides of (\ref{regular1}), we obtain
	\begin{align*}
		\int_{0}^{\infty}e^{-zt}\sum_{n=0}^{\infty}\bar{p}^\nu(n,t)\,\mathrm{d}t&=\frac{z^{\nu-1}}{z^\nu+\alpha}+\sum_{n=1}^{\infty}\sum_{k=1}^{n}\binom{n-1}{k-1}\frac{(-1)^{k}\alpha}{k\lambda-\alpha}\left(\frac{z^{\nu-1}}{z^\nu+k\lambda}-\frac{z^{\nu-1}}{z^\nu+\alpha}\right)\\
		&=\frac{z^{\nu-1}}{z^\nu+\alpha}-\frac{z^{\nu-1}\alpha}{(z^\nu+\alpha)\lambda}\sum_{n=1}^{\infty}\sum_{k=1}^{n}\binom{n-1}{k-1}\frac{(-1)^{k}}{z^\nu/\lambda+k}\\
		&=\frac{z^{\nu-1}}{z^\nu+\alpha}+\frac{z^{\nu-1}\alpha}{(z^\nu+\alpha)\lambda}\sum_{n=1}^{\infty}\sum_{k=0}^{n-1}\binom{n-1}{k}\frac{(-1)^{k}}{z^\nu/\lambda+1+k}.
	\end{align*}
	Now, we use the following result (see \cite{Kirschenhofer1996}):
	\begin{equation}\label{formula}
		\sum_{i=0}^{N}\binom{N}{i}\frac{(-1)^i}{a+i}=\frac{N!}{a(a+1)\cdots(a+N)},\ a\in\mathbb{R},
	\end{equation}
	to obtain
	\begin{align}
		\int_{0}^{\infty}e^{-zt}\sum_{n=0}^{\infty}\bar{p}^\nu(n,t)\,\mathrm{d}t&=\frac{z^{\nu-1}}{z^\nu+\alpha}+\frac{z^{\nu-1}\alpha}{(z^\nu+\alpha)\lambda}\sum_{n=1}^{\infty}\frac{(n-1)!}{(z^\nu/\lambda+1)(z^\nu/\lambda+2)\cdots(z^\nu/\lambda+n)}\nonumber\\
		&=\frac{z^{\nu-1}}{z^\nu+\alpha}+\frac{z^{\nu-1}\alpha}{(z^\nu+\alpha)\lambda}\sum_{k=0}^{\infty}\frac{\Gamma(k+1)\Gamma(z^\nu/\lambda+1)}{\Gamma(z^\nu/\lambda+1+(k+1))}\nonumber\\
		&=\frac{z^{\nu-1}}{z^\nu+\alpha}+\frac{z^{\nu-1}\alpha}{(z^\nu+\alpha)\lambda}\sum_{k=0}^{\infty}B(k+1,z^\nu/\lambda+1)\nonumber\\
		&=\frac{z^{\nu-1}}{z^\nu+\alpha}+\frac{z^{\nu-1}\alpha}{(z^\nu+\alpha)\lambda}\sum_{k=0}^{\infty}\int_{0}^{1}w^k(1-w)^{z^\nu/\lambda}\,\mathrm{d}w\label{pfbt}\\
		&=\frac{z^{\nu-1}}{z^\nu+\alpha}+\frac{z^{\nu-1}\alpha}{(z^\nu+\alpha)\lambda}\int_{0}^{1}(1-w)^{z^\nu/\lambda-1}\,\mathrm{d}w\nonumber\\
		&=\frac{z^{\nu-1}}{z^\nu+\alpha}+\frac{\alpha}{z^\nu+\alpha}\int_{0}^{1}\frac{1}{z}\,\mathrm{d}x=\frac{1}{z},\ \ (\text{taking $(1-w)^{z^\nu/\lambda}=x$})\nonumber
	\end{align}
	where we have used $B(a,b)=\int_{0}^{1}w^{a-1}(1-w)^{b-1}\,\mathrm{d}w$, $a>0$, $b>0$ to get (\ref{pfbt}). This completes the proof.
\end{proof}
\begin{remark}
	The pgf $\bar{G}^\nu(u,t)=\mathbb{E}u^{\bar{N}^\nu(t)}$, $|u|\leq1$ of TFLBPwI solves 
	\begin{equation}\label{pgfflbpwi}
		\frac{\partial^\nu}{\partial t^\nu}\bar{G}^\nu(u,t)=\lambda u(u-1)\frac{\partial}{\partial u}\bar{G}^\nu(u,t)+\alpha(u-1)\bar{p}^\nu(0,t)
	\end{equation}
	with $\bar{G}^\nu(u,0)=1$. On taking the derivative of (\ref{pgfflbpwi}) with respect to $u$ and substituting $u=1$, we get
	$
		({\mathrm{d}^\nu}/{\mathrm{d}t^\nu})	\mathbb{E}\bar{N}^\nu(t)=\lambda\mathbb{E}\bar{N}^\nu(t)+\alpha \bar{p}^\nu(0,t)\ \text{with}\ 	\mathbb{E}\bar{N}^\nu(0)=0.
	$
	So, using (\ref{frderlap}) and (\ref{p01lap}), the Laplace transform of $	\mathbb{E}\bar{N}^\nu(t)$ is
	\begin{equation}\label{lapmeanc}
		\int_{0}^{\infty}e^{-zt}	\mathbb{E}\bar{N}^\nu(t)\,\mathrm{d}t=\frac{\alpha}{\alpha+\lambda}\left(\frac{z^{\nu-1}}{z^\nu-\lambda}-\frac{z^{\nu-1}}{z^\nu+\alpha}\right).
	\end{equation}
	On taking the inverse Laplace transform, we get 
$
			\mathbb{E}\bar{N}^\nu(t)={\alpha}{(\alpha+\lambda)^{-1}}(E_{\nu,1}(t^\nu\lambda)-E_{\nu,1}({-t^\nu\alpha})).
$

	 Similarly, the second factorial moment $\bar{\mu}^\nu(t)\coloneqq\mathbb{E}\bar{N}^\nu(t)(\bar{N}^\nu(t)-1)=({\partial^2}/{\partial u^2})\bar{G}^\nu(u,t)\big|_{u=1}$ of TFLBPwI solves
	\begin{equation}\label{00var}
		\frac{\mathrm{d}^\nu \bar{\mu}^\nu(t)}{\mathrm{d}t^\nu}=2\lambda\mathbb{E}\bar{N}^\nu(t)+2\lambda\bar{\mu}^\nu(t),\  \bar{\mu}^\nu(0)=0.
	\end{equation}
	On taking the Laplace transform on both sides of (\ref{00var}) and using (\ref{lapmeanc}), we have
	\begin{align*}
		\int_{0}^{\infty}e^{-zt}\bar{\mu}^\nu(t)\,\mathrm{d}t&=\frac{2\alpha\lambda}{\alpha+\lambda}\Bigg(\frac{1}{\lambda}\left(\frac{z^{\nu-1}}{z^\nu-2\lambda}-\frac{z^{\nu-1}}{z^\nu-\lambda}\right) -\frac{1}{\alpha+2\lambda}\left(\frac{z^{\nu-1}}{z^\nu-2\lambda}-\frac{z^{\nu-1}}{z^\nu+\alpha}\right)\Bigg)\\
	&=\frac{2\alpha\lambda}{\alpha+\lambda}\Bigg(\frac{\alpha+\lambda}{\lambda(\alpha+2\lambda)}\frac{z^{\nu-1}}{z^\nu-2\lambda} -\frac{1}{\lambda}\frac{z^{\nu-1}}{z^\nu-\lambda}+\frac{1}{\alpha+2\lambda}\frac{z^{\nu-1}}{z^\nu+\alpha}\Bigg).
\end{align*}
		Its inverse Laplace transform gives
		\begin{equation*}
			\bar{\mu}^\nu(t)=\frac{2\alpha\lambda}{\alpha+\lambda}\left(\frac{(\alpha+\lambda)E_{\nu,1}(2t^\nu\lambda)}{\lambda(\alpha+2\lambda)}-\frac{E_{\nu,1}(t^\nu\lambda)}{\lambda}+\frac{E_{\nu,1}(-t^\nu\alpha)}{\alpha+2\lambda}\right).
		\end{equation*}
		Thus, the variance of TFLBPwI is
		\begin{align*}
			\mathbb{V}\mathrm{ar}\bar{N}^\nu(t)&=\bar{\mu}^\nu(t)+\mathbb{E}\bar{N}^\nu(t)-(\mathbb{E}\bar{N}^\nu(t))^2\\
			&=\frac{2\alpha E_{\nu,1}(2t^\nu\lambda)}{\alpha+2\lambda}-\frac{\alpha E_{\nu,1}(t^\nu\lambda)}{\alpha+\lambda}-\frac{\alpha^2E_{\nu,1}(-t^\nu\alpha)}{(\alpha+\lambda)(\alpha+2\lambda)} -\frac{\alpha^2}{(\alpha+\lambda)^2}(E_{\nu,1}(t^\nu\lambda)-E_{\nu,1}({-t^\nu\alpha}))^2.
		\end{align*}
\end{remark}

\begin{theorem}
	The Laplace transform of the pgf of TFLBPwI is given by
	\begin{equation*}
	\tilde{\bar{G}}^\nu(u,z)=\frac{z^{\nu-1}}{z^\nu+\alpha}+\frac{\alpha }{(z^\nu+\alpha)}\int_{0}^{\infty}\frac{ue^{-\lambda x}}{1-u(1-e^{-\lambda x})}e^{-xz^\nu}z^{\nu-1}\,\mathrm{d}x,\ z>0.
	\end{equation*}
\end{theorem}
\begin{proof}
	On using (\ref{00p0}) and (\ref{Mittag-lap}), we have
	\begin{align*}
		\tilde{\bar{G}}^\nu(u,z)&=\int_{0}^{\infty}e^{-zt}{\bar{G}}^\nu(u,t)\,\mathrm{d}t\\
		&=\frac{z^{\nu-1}}{z^\nu+\alpha}-\frac{z^{\nu-1}\alpha}{z^\nu+\alpha}\sum_{n=1}^{\infty}u^n\sum_{k=1}^{n}\binom{n-1}{k-1}\frac{(-1)^k}{z^\nu+k\lambda}\\
		&=\frac{z^{\nu-1}}{z^\nu+\alpha}+\frac{z^{\nu-1}\alpha}{(z^\nu+\alpha)\lambda}\sum_{n=1}^{\infty}u^n\sum_{k=0}^{n-1}\binom{n-1}{k}\frac{(-1)^k}{z^\nu/\lambda+1+k}\\
		&=\frac{z^{\nu-1}}{z^\nu+\alpha}+\frac{z^{\nu-1}\alpha}{(z^\nu+\alpha)\lambda}\sum_{n=1}^{\infty}u^n\frac{(n-1)!}{(z^\nu/\lambda+1)(z^\nu/\lambda+2)\dots(z^\nu/\lambda+n)},\ \ \text{(using (\ref{formula}))}\\
		&=\frac{z^{\nu-1}}{z^\nu+\alpha}+\frac{z^{\nu-1}\alpha u}{(z^\nu+\alpha)\lambda}\sum_{j=0}^{\infty}u^jB(j+1,z^\nu/\lambda+1)\\
		&=\frac{z^{\nu-1}}{z^\nu+\alpha}+\frac{z^{\nu-1}\alpha u}{(z^\nu+\alpha)\lambda}\int_{0}^{1}\sum_{j=0}^{\infty}u^jw^j(1-w)^{z^\nu/\lambda}\,\mathrm{d}w\\
		&=\frac{z^{\nu-1}}{z^\nu+\alpha}+\frac{z^{\nu-1}\alpha u}{(z^\nu+\alpha)\lambda}\int_{0}^{1}\frac{(1-w)^{z^\nu/\lambda}}{(1-uw)}\,\mathrm{d}w.
	\end{align*}
	On substituting $ 1-w=e^{-\lambda x}$, we get the required result.
\end{proof}

\subsection{Time-fractional linear death process with immigration} Suppose we allow $\lambda=0$ in LBDPwI then  its time-changed version \textit{viz} TFLBDPwI reduces to the time-fractional linear death process with immigration (TFLDPwI). We denote it by $\{\hat{N}^\nu(t)\}_{t\ge0}$. So, the TFLDPwI can attain only two possible states that are state $0$ and state $1$. For any $t\ge0$, the marginal distribution of $\hat{N}^\nu(t)$ is Bernoulli with success probability $\hat{p}^\nu(1,t)=\mathrm{Pr}\{\hat{N}^\nu(t)=1\}$. The state probabilities of TFLDPwI solve
\begin{equation}\label{fidpequ}
		\frac{\mathrm{d}^\nu }{\mathrm{d}t^\nu}\hat{p}^\nu(n,t)=\begin{cases}
			-\alpha \hat{p}^\nu(0,t)+\mu \hat{p}^\nu(1,t),\ n=0,\vspace{0.1cm}\\
			-\mu\hat{ p}^\nu(1,t)+\alpha \hat{p}^\nu(0,t),\ n=1,
		\end{cases}
\end{equation}
with $\hat{p}^\nu(0,0)=1$.
Let $g(t)=({\mathrm{d}^\nu}/{\mathrm{d}t^\nu})\hat{p}^\nu(0,t)$. On taking the Caputo fractional derivative on both sides of (\ref{fidpequ}) and using $({\mathrm{d}^\nu}/{\mathrm{d}t^\nu})\hat{p}^\nu(0,t)=-({\mathrm{d}^\nu}/{\mathrm{d}t^\nu})\hat{p}^\nu(1,t)$, we get
$
	({\mathrm{d}^\nu}/{\mathrm{d}t^\nu})g(t)=-(\alpha+\mu)g(t)
$
with initial condition $g(0)=-\alpha.$
By using (\ref{frderlap}), the Laplace transform of ${g}(t)$ is $
\tilde{g}(z)={-\alpha z^{\nu-1}}/{(z^\nu+\alpha+\mu)}.
$ Its inverse Laplace transform gives
\begin{equation}\label{invgt}
	\frac{\mathrm{d}^\nu \hat{p}^\nu(0,t)}{\mathrm{d}t^\nu}=-\alpha E_{\nu,1}(-t^\nu(\alpha+\mu))
\end{equation}
with $\hat{p}^\nu(0,0)=1$.
On taking the Laplace transform on both sides of (\ref{invgt}), we get
\begin{equation*}
	\tilde{\hat{p}}^\nu(0,z)=\frac{1}{z}-\frac{\alpha}{\alpha+\mu}\left(\frac{1}{z}-\frac{z^{\nu-1}}{z^\nu+\alpha+\mu}\right).
\end{equation*}
By taking inverse Laplace transform, we obtain
\begin{equation}\label{p00}
\hat{p}^\nu(0,t)=\frac{\alpha E_{\nu,1}(-t^\nu(\alpha+\mu))+\mu}{\alpha+\mu}.
\end{equation}
Now, on substituting (\ref{invgt}) and (\ref{p00}) in (\ref{fidpequ}), we get
\begin{equation*}
	\hat{p}^\nu(1,t)=\frac{\alpha(1-E_{\nu,1}(-t^\nu(\alpha+\mu)))}{\alpha+\mu}.
\end{equation*}
\begin{remark}
	For $\nu=1$, the TFLDPwI reduces to a linear death process with immigration. In this case, the state probabilities are given by
	\begin{equation*}
		\hat{p}(0,t)=
			\frac{\alpha e^{-t(\alpha+\mu)}+\mu}{\alpha+\mu}\ \ \text{and}\ \ 
			\hat{p}(1,t)=\frac{\alpha(1-e^{-t(\alpha+\mu)})}{\alpha+\mu}.		
	\end{equation*}
\end{remark}

\end{document}